\documentclass{amsart}
\usepackage{amsfonts}
\usepackage{mathrsfs}
\usepackage{amsthm}
\usepackage{amssymb}
\usepackage{amsmath}
\usepackage{enumerate}

\theoremstyle{plain}
\newtheorem{theorem}{Theorem}[section]
\newtheorem{proposition}[theorem]{Proposition}

\newtheorem{lemma}[theorem]{Lemma}

\theoremstyle{definition}
\newtheorem{definition}{Definition}[section]

\theoremstyle{remark}

\theoremstyle{example}

\numberwithin{equation}{section}

\begin{document}

\title[Convergence Theorems]{Convergence Theorems for Generalized Functional Sequences of Discrete-Time Normal Martingales}

\author{Caishi Wang}
\address[Caishi Wang]
          {School of Mathematics and Statistics,
          Northwest Normal University,
          Lanzhou, Gansu 730070,
          People's Republic of China }
\email{wangcs@nwnu.edu.cn}

\author{Jinshu Chen}
\address[Jinshu Chen]
         {School of Mathematics and Statistics,
         Northwest Normal University,
          Lanzhou, Gansu 730070,
          People's Republic of China}


\begin{abstract}
The Fock transform recently introduced by the authors in a previous paper is applied to investigate convergence
of generalized functional sequences of a discrete-time normal martingale $M$.
A necessary and sufficient condition in terms of the Fock transform is obtained for such a sequence to be strong convergent.
A type of generalized martingales associated with $M$ are introduced and their convergence theorems are established.
Some applications are also shown.
\end{abstract}

\maketitle

\section{Introduction}\label{sec-1}

Hida's white noise analysis is essentially a theory of infinite dimensional calculus on generalized functionals
of Brownian motion \cite{hida,huang,kuo,obata}. In 1988, Ito \cite{ito} introduced his analysis of generalized Poisson functionals, which can be viewed as
a theory of infinite dimensional calculus on generalized functionals of Poisson martingale.
It is known that both Brownian motion and Poisson martingale are continuous-time normal martingales.
There are theories of white noise analysis for some other continuous-time processes (see, e.g., \cite{albe, barhoumi, di, hu,lee}).

Discrete-time normal martingales \cite{privault} also play an important role in many theoretical and applied fields.
For example, the classical random walk (a special discrete-time normal martingale) is used to establish functional central limit theorems
in probability theory \cite{dud,rud}.
It would then be interesting to develop a theory of infinite dimensional calculus on generalized functionals
of discrete-time normal martingales.

Let $M=(M_n)_{n\in \mathbb{N}}$ be a discrete-time normal martingale satisfying some mild conditions.
In a recent paper \cite{wang-chen}, we constructed generalized functionals of $M$, and introduced a transform, called the Fock transform,
to characterize those functionals.

In this paper, we apply the Fock transform \cite{wang-chen} to investigate generalized functional sequences of $M$.
First, by using the Fock transform, we obtain a necessary and sufficient condition for a generalized functional sequence of $M$ to be strong convergent.
Then we introduce a type of generalized martingales associated with $M$, called $M$-generalized martingales, which are a special class of generalized functional sequences of $M$
and include as a special case the classical martingales with respect to the filtration generated by $M$.
We establish a strong-convergent criterion in terms of the Fock transform for $M$-generalized martingales. Some other convergence criteria are also obtained.
Finally we show some applications of our main results.

Our one interesting finding is that for an $M$-generalized martingale, its strong convergence is just equivalent to its strong boundedness.

Throughout this paper, $\mathbb{N}$ designates the set of all nonnegative integers and $\Gamma$ the finite power set of $\mathbb{N}$, namely
\begin{equation}\label{eq-1-1}
\Gamma = \{\,\sigma \mid \text{$\sigma \subset \mathbb{N}$ and $\#(\sigma) < \infty$} \,\},
\end{equation}
where $\#(\sigma)$ means the cardinality of $\sigma$ as a set.
In addition, we always assume that $(\Omega, \mathcal{F}, P)$ is a given probability space with $\mathbb{E}$ denoting the expectation with respect to $P$.
We denote by $\mathcal{L}^{2}(\Omega, \mathcal{F}, P)$ the usual Hilbert space of
square integrable complex-valued functions on $(\Omega, \mathcal{F}, P)$
and use $\langle\cdot,\cdot\rangle$ and $\|\cdot\|$ to mean its inner product and norm, respectively.
By convention, $\langle\cdot,\cdot\rangle$ is conjugate-linear in its first argument and linear in its second argument.

\section{Generalized functionals}\label{sec-2}

Let $M=(M_n)_{n\in \mathbb{N}}$ be a discrete-time normal martingale on $(\Omega, \mathcal{F}, P)$
that has the chaotic representation property and
$Z=(Z_n)_{n\in \mathbb{N}}$ the discrete-time normal noise associated
with $M$ (see Appendix). We define
\begin{equation}\label{eq-2-1}
   Z_{\emptyset}=1;\quad Z_{\sigma} = \prod_{i\in \sigma}Z_i,\quad \text{$\sigma \in \Gamma$, $\sigma \neq \emptyset$}.
\end{equation}
And, for brevity, we use $\mathcal{L}^2(M)$ to mean the space of square integrable functionals of $M$, namely
\begin{equation}\label{eq-2-2}
  \mathcal{L}^2(M) = \mathcal{L}^2(\Omega, \mathcal{F}_{\infty}, P),
\end{equation}
which shares the same inner product and norm with $\mathcal{L}^2(\Omega, \mathcal{F}, P)$, namely $\langle\cdot,\cdot\rangle$ and $\|\cdot\|$.
We note that $\{Z_{\sigma}\mid \sigma \in \Gamma\}$ forms a countable orthonormal basis for $\mathcal{L}^2(M)$ (see Appendix).

\begin{lemma}\label{lem-2-1}\cite{wang-z}
Let $\sigma\mapsto\lambda_{\sigma}$ be the $\mathbb{N}$-valued function on $\Gamma$ given by
\begin{equation}\label{eq-2-3}
\lambda_{\sigma}=
\left\{
  \begin{array}{ll}
    \prod_{k\in\sigma}(k+1), & \hbox{$\sigma\neq \emptyset$, $\sigma\in\Gamma$;}\\
    1, & \hbox{$\sigma=\emptyset$, $\sigma\in\Gamma$.}
  \end{array}
\right.
\end{equation}
Then, for $p>1$, the positive term series $\sum_{\sigma\in\Gamma}\lambda^{-p}_{\sigma}$ converges and moreover
\begin{equation}\label{eq-2-4}
\sum_{\sigma\in\Gamma}\lambda^{-p}_{\sigma}\leq \exp\bigg[\sum_{k=1}^{\infty}k^{-p}\bigg]<\infty.
\end{equation}
\end{lemma}

Using the $\mathbb{N}$-valued function defined by (\ref{eq-2-3}), we can construct a chain of Hilbert spaces consisting of functionals of $M$ as follows.
For $p\geq 0$, we define a norm $\|\cdot\|_p$ on $\mathcal{L}^2(M)$ through
\begin{equation}\label{eq-2-5}
  \|\xi\|_{p}^2=\sum_{\sigma\in \Gamma}\lambda_{\sigma}^{2p}|\langle Z_{\sigma}, \xi\rangle|^{2},\quad \xi \in \mathcal{L}^2(M)
\end{equation}
and put
\begin{equation}\label{eq-2-6}
  \mathcal{S}_p(M) = \big\{\, \xi \in \mathcal{L}^2(M) \mid \|\xi\|_{p}< \infty\,\big\}.
\end{equation}
It is not hard to check that $\|\cdot\|_{p}$ is a Hilbert norm and $\mathcal{S}_p(M)$ becomes a Hilbert space
with $\|\cdot\|_{p}$. Moreover, the inner product corresponding to $\|\cdot\|_{p}$ is given by
\begin{equation}\label{eq-2-7}
  \langle \xi,\eta\rangle_p
  = \sum_{\sigma\in \Gamma}\lambda_{\sigma}^{2p}\overline{\langle Z_{\sigma},\xi\rangle} \langle Z_{\sigma}, \eta\rangle,\quad
  \xi,\, \eta \in \mathcal{S}_p(M).
\end{equation}
Here $\overline{\langle Z_{\sigma},\xi\rangle}$ means the complex conjugate of $\langle Z_{\sigma},\xi\rangle$.

\begin{lemma}\label{lem-2-2}\cite{wang-chen}
For each $p\geq 0$, one has $\{Z_{\sigma}\mid \sigma\in\Gamma\} \subset \mathcal{S}_p(M)$ and moreover the system
$\{\lambda^{-p}_{\sigma}Z_{\sigma}\mid \sigma\in\Gamma\}$ forms an orthonormal basis for $\mathcal{S}_p(M)$.
\end{lemma}

It is easy to see that $\lambda_{\sigma}\geq 1$ for all $\sigma\in \Gamma$. This implies that $\|\cdot\|_p \leq \|\cdot\|_q$
and $\mathcal{S}_q(M)\subset \mathcal{S}_p(M)$
whenever $0\leq p \leq q$. Thus we actually get a chain of Hilbert spaces of functionals of $M$:
\begin{equation}\label{eq-2-8}
 \cdots \subset \mathcal{S}_{p+1}(M) \subset \mathcal{S}_p(M)\subset  \cdots \subset \mathcal{S}_1(M) \subset \mathcal{S}_0(M)=\mathcal{L}^2(M).
\end{equation}
We now put
\begin{equation}\label{eq-2-9}
  \mathcal{S}(M)=\bigcap ^{\infty}_{p=0}\mathcal{S}_{p}(M)
\end{equation}
and endow it with the topology generated by the norm sequence $\{\|\cdot \|_{p}\}_{p\geq 0}$.
Note that, for each $ p\geq 0$, $\mathcal{S}_p(M)$ is just the completion of $\mathcal{S}(M)$ with respect to $\|\cdot\|_{p}$.
Thus $\mathcal{S}(M)$ is a countably-Hilbert space \cite{becnel, gelfand-vi}.
The next lemma, however, shows that $\mathcal{S}(M)$ even has a much better property.

\begin{lemma}\label{lem-2-3}\cite{wang-chen}
The space $\mathcal{S}(M)$ is a nuclear space, namely for any $p\geq 0$,
there exists $q> p$ such that the inclusion mapping $i_{pq}\colon \mathcal{S}_{q}(M) \rightarrow \mathcal{S}_p(M)$
defined by $i_{pq}(\xi)=\xi$ is a Hilbert-Schmidt operator.
\end{lemma}

For $p\geq 0$, we denote by $\mathcal{S}_p^*(M)$ the dual of $\mathcal{S}_p(M)$ and $\|\cdot\|_{-p}$
the norm of $\mathcal{S}_p^*(M)$. Then $\mathcal{S}_p^*(M)\subset \mathcal{S}_q^*(M)$ and
$\|\cdot\|_{-p} \geq \|\cdot\|_{-q}$ whenever $0\leq p \leq q$.
The lemma below is then an immediate consequence of the general theory of countably-Hilbert spaces (see, e.g., \cite{becnel} or \cite{gelfand-vi}).

\begin{lemma}\label{lem-2-4}\cite{wang-chen}
Let $\mathcal{S}^*(M)$ the dual of $\mathcal{S}(M)$ and endow it with the strong topology. Then
\begin{equation}\label{eq-2-10}
  \mathcal{S}^*(M)=\bigcup_{p=0}^{\infty}\mathcal{S}_p^*(M)
\end{equation}
and moreover the inductive limit topology on $\mathcal{S}^*(M)$ given by space sequence $\{\mathcal{S}_p^*(M)\}_{p\geq 0}$ coincides with the strong topology.
\end{lemma}

We mention that, by identifying $\mathcal{L}^2(M)$ with its dual, one comes to a Gel'fand triple
\begin{equation}\label{eq-2-11}
\mathcal{S}(M)\subset \mathcal{L}^2(M)\subset \mathcal{S}^*(M),
\end{equation}
which we refer to as the Gel'fand triple associated with $M$.

\begin{lemma}\label{lem-2-5}\cite{wang-chen}
The system $\{Z_{\sigma} \mid \sigma \in \Gamma\}$ is contained in $\mathcal{S}(M)$ and moreover it serves as a basis in $\mathcal{S}(M)$ in the sense that
\begin{equation}\label{eq-2-12}
  \xi = \sum_{\sigma \in \Gamma} \langle Z_{\sigma}, \xi\rangle Z_{\sigma}, \quad \xi \in \mathcal{S}(M),
\end{equation}
where $\langle\cdot,\cdot\rangle$ is the inner product of $\mathcal{L}^2(M)$ and the series converges in the topology of $\mathcal{S}(M)$.
\end{lemma}

\begin{definition}\label{def-2-1}\cite{wang-chen}
Elements of $\mathcal{S}^*(M)$ are called generalized functionals of $M$, while elements of $\mathcal{S}(M)$ are called testing functionals of $M$.
\end{definition}

Denote by $\langle\!\langle \cdot,\cdot\rangle\!\rangle$
the canonical bilinear form on $\mathcal{S}^*(M)\times \mathcal{S}(M)$, namely
\begin{equation}\label{eq-2-13}
  \langle\!\langle \Phi,\xi\rangle\!\rangle = \Phi(\xi),\quad \Phi\in \mathcal{S}^*(M),\, \xi\in \mathcal{S}(M),
\end{equation}
where $\Phi(\xi)$ means $\Phi$ acting on $\xi$ as usual.
Note that $\langle\cdot,\cdot\rangle$ denotes the inner product of $\mathcal{L}^2(M)$, which is different from
$\langle\!\langle \cdot,\cdot\rangle\!\rangle$.

\begin{definition}\label{def-2-2}\cite{wang-chen}
For $\Phi \in \mathcal{S}^*(M)$, its Fock transform is the function $\widehat{\Phi}$ on $\Gamma$ given by
\begin{equation}\label{eq-2-14}
  \widehat{\Phi}(\sigma) = \langle\!\langle \Phi, Z_{\sigma}\rangle\!\rangle,\quad \sigma \in \Gamma,
\end{equation}
where $\langle\!\langle \cdot,\cdot\rangle\!\rangle$ is the canonical bilinear form.
\end{definition}

It is easy to to verify that,
for $\Phi$, $\Psi \in \mathcal{S}^*(M)$, $\Phi=\Psi$ if and only if $\widehat{\Phi}=\widehat{\Psi}$.
Thus a generalized functional of $M$ is completely determined by its Fock transform.
The following theorem characterizes generalized functionals of $M$ through their Fock transforms.

\begin{lemma}\label{lem-2-6}\cite{wang-chen}
Let $F$ be a function on $\Gamma$. Then $F$ is the Fock transform of an element $\Phi$ of $\mathcal{S}^*(M)$ if and only if it satisfies
\begin{equation}\label{eq-2-15}
  |F(\sigma)| \leq C\lambda_{\sigma}^p,\quad \sigma \in \Gamma
\end{equation}
for some constants $C\geq 0$ and $p\geq 0$.
In that case, for $q> p+\frac{1}{2}$, one has
\begin{equation}\label{eq-2-16}
  \|\Phi\|_{-q} \leq C\bigg[\sum_{\sigma \in \Gamma}\lambda_{\sigma}^{-2(q-p)}\bigg]^{\frac{1}{2}}
\end{equation}
and in particular $\Phi \in \mathcal{S}_q^*(M)$.
\end{lemma}

\section{Convergence theorems for generalized functional sequences}\label{sec-3}

Let $M=(M_n)_{n\in \mathbb{N}}$ be the same discrete-time normal martingale as described in Section~\ref{sec-2}.
In the present section, we apply the Fock transform (see Definition~\ref{def-2-2}) to establish convergence theorems for generalized functionals of $M$.

In order to prove our main results in a convenient way, we first give a preliminary proposition,
which is an immediate consequence of the general theory of countably normed spaces, especially nuclear spaces \cite{becnel, gelfand-shi, gelfand-vi},
since $\mathcal{S}(M)$ is a nuclear space (see Lemma~\ref{lem-2-3}).

\begin{proposition}\label{prop-3-1}
Let $\Phi$, $\Phi_n \in \mathcal{S}^*(M)$, $n\geq 1$, be generalized functionals of $M$. Then the following conditions are equivalent:

(i)\ \ The sequence $(\Phi_n)$ converges weakly to $\Phi$ in $\mathcal{S}^*(M)$;

(ii)\ \ The sequence $(\Phi_n)$ converges strongly to $\Phi$ in $\mathcal{S}^*(M)$;

(iii)\ \ There exists a constant $p\geq 0$ such that $\Phi$, $\Phi_n \in \mathcal{S}_p^*(M)$, $n\geq 1$,
and the sequence $(\Phi_n)$ converges to $\Phi$ in the norm of $\mathcal{S}_p^*(M)$.
\end{proposition}

Here we mention that ``$(\Phi_n)$ converges strongly (resp. weakly) to $\Phi$'' means that $(\Phi_n)$ converges to $\Phi$ in the strong (resp. weak) topology of $\mathcal{S}^*(M)$.
For details about various topologies on the dual of a countably normed space, we refer to \cite{becnel, gelfand-shi}.

The next theorem is one of our main results, which offers a criterion in terms of the Fock transform for checking whether or not a sequence in $\mathcal{S}^*(M)$ is strongly convergent.

\begin{theorem}\label{thr-3-1}
Let $\Phi$, $\Phi_n \in \mathcal{S}^*(M)$, $n\geq 1$, be generalized functionals of $M$.  Then the sequence $(\Phi_n)$ converges strongly to $\Phi$ in $\mathcal{S}^*(M)$ if and only if it satisfies:
\begin{enumerate}
  \item[(1)] $\widehat{\Phi_n}(\sigma)\rightarrow \widehat{\Phi}(\sigma)$ for all $\sigma\in \Gamma$;
  \item[(2)] There are constants $C\geq 0$ and $p\geq 0$ such that
\begin{equation}\label{}
   \sup_{n\geq 1}|\widehat{\Phi_n}(\sigma)| \leq C\lambda_{\sigma}^p,\quad \sigma \in \Gamma.
\end{equation}
\end{enumerate}
\end{theorem}

\begin{proof}
The ``only if '' part. Let $(\Phi_n)$ converge strongly to $\Phi$ in $\mathcal{S}^*(M)$. Then, we obviously have
\begin{equation*}
\widehat{\Phi_n}(\sigma)=\langle\!\langle \Phi_n,Z_{\sigma}\rangle\!\rangle
\rightarrow
\langle\!\langle \Phi,Z_{\sigma}\rangle\!\rangle=\widehat{\Phi}(\sigma),\quad \sigma\in \Gamma,
\end{equation*}
because $\{Z_{\sigma}\mid \sigma \in \Gamma\}\subset \mathcal{S}(M)$ and $(\Phi_n)$ also converges weakly to $\Phi$.
On the other hand, by Proposition~\ref{prop-3-1}, we know that there exists $p\geq 0$ such that
$\Phi$, $\Phi_n \in \mathcal{S}_p^*(M)$, $n\geq 1$,
and $(\Phi_n)$ converges to $\Phi$ in the norm of $\mathcal{S}_p^*(M)$,
which implies that $C\equiv \sup_{n\geq 1}\|\Phi_n\|_{-p} < \infty$.
Therefore
\begin{equation*}
  \sup_{n\geq 1}|\widehat{\Phi_n}(\sigma)|
   = \sup_{n\geq 1}|\langle\!\langle \Phi_n,Z_{\sigma}\rangle\!\rangle|
   \leq \sup_{n\geq 1}\|\Phi_n\|_{-p}\|Z_{\sigma}\|_p
   = C\lambda_{\sigma}^p,\quad \sigma \in \Gamma.
\end{equation*}

The ``if'' part. Let $(\Phi_n)$ satisfy conditions (1) and (2). Then,
by taking $q > p+ \frac{1}{2}$ and using Lemma~\ref{lem-2-6}, we get
\begin{equation}\label{eq-3-2}
  \sup_{n\geq 1}\|\Phi_n\|_{-q} \leq C\bigg[\sum_{\sigma \in \Gamma}\lambda_{\sigma}^{-2(q-p)}\bigg]^{\frac{1}{2}},
\end{equation}
in particular $\Phi_n \in \mathcal{S}_q^*(M)$, $n\geq 1$. On the other hand, $\{Z_{\sigma} \mid \sigma \in \Gamma\}$ is total in $\mathcal{S}_q(M)$,
which, together with (\ref{eq-3-2}) as well as the property
\begin{equation*}
  \langle\!\langle \Phi_n,Z_{\sigma}\rangle\!\rangle
   =\widehat{\Phi_n}(\sigma)\rightarrow \widehat{\Phi}(\sigma)
   = \langle\!\langle \Phi,Z_{\sigma}\rangle\!\rangle,\quad \sigma\in \Gamma,
\end{equation*}
implies that
$\Phi \in \mathcal{S}_q^*(M)$ and
\begin{equation*}
\langle\!\langle \Phi_n,\xi\rangle\!\rangle \rightarrow \langle\!\langle \Phi,\xi\rangle\!\rangle,\quad \forall\, \xi\in \mathcal{S}_q(M).
\end{equation*}
Thus $(\Phi_n)$ converges weakly to $\Phi$ in $\mathcal{S}^*(M)$, which together with Proposition~\ref{prop-3-1} implies that
$(\Phi_n)$ converges strongly to $\Phi$ in $\mathcal{S}^*(M)$.
\end{proof}

In a similar way we can prove the following theorem, which is slightly different form Theorem~\ref{thr-3-1}, but more convenient to use.

\begin{theorem}\label{thr-3-2}
Let $(\Phi_n) \subset \mathcal{S}^*(M)$ be a sequence of generalized functionals of $M$.  Suppose
$\big(\widehat{\Phi_n}(\sigma)\big)$ converges for all $\sigma\in \Gamma$, and moreover there are constants $C\geq 0$ and $p\geq 0$ such that
\begin{equation}\label{}
   \sup_{n\geq 1}|\widehat{\Phi_n}(\sigma)| \leq C\lambda_{\sigma}^p,\quad \sigma \in \Gamma.
\end{equation}
Then there exists a generalized functional $\Phi\in \mathcal{S}^*(M)$ such that $(\Phi_n)$ converges strongly to $\Phi$.
\end{theorem}

\section{$M$-generalized martingales and their convergence theorems}

In this section, we first introduce a type of generalized martingales associated with $M$,
which we call $M$-generalized martingales, and then we use the Fock transform to a give necessary and sufficient condition for such a generalized martingale to be strongly convergent. Some other convergence results are also obtained.

For a nonnegative integer $n\geq 0$, we denote by $\Gamma\!_{n]}$ the power set of $\{0,1,\cdots, n\}$, namely
\begin{equation}\label{eq-4-1}
  \Gamma\!_{n]}=\{\,\sigma \mid \sigma\subset \{0,1, \cdots, n\}\,\}.
\end{equation}
Clearly $\Gamma\!_{n]} \subset \Gamma$. We use $\mathbf{I}_{n]}$ to mean the indicator of $\Gamma\!_{n]}$, which is a function on $\Gamma$ given by
\begin{equation}\label{eq-4-2}
  \mathbf{I}_{n]}(\sigma)=
  \left\{
    \begin{array}{ll}
      1, & \hbox{$\sigma \in \Gamma\!_{n]}$;} \\
      0, & \hbox{$\sigma \notin \Gamma\!_{n]}$.}
    \end{array}
  \right.
\end{equation}

\begin{definition}\label{def-4-1}
A sequence $(\Phi_n)_{n\geq 0} \subset \mathcal{S}^*(M)$ is called an $M$-generalized martingale if it satisfies that
\begin{equation}\label{eq-4-3}
  \widehat{\Phi_n}(\sigma) = \mathbf{I}_{n]}(\sigma)\widehat{\Phi_{n+1}}(\sigma),\quad \sigma \in \Gamma,\, n\geq 0,
\end{equation}
where $\mathbf{I}_{n]}$ mean the indicator of $\Gamma\!_{n]}$ as defined by (\ref{eq-4-2}).
\end{definition}

Let $\mathfrak{F}=(\mathcal{F}_n)_{n\geq 0}$ be the filtration on $(\Omega, \mathcal{F}, P)$  generated by $Z=(Z_n)_{n\geq 0}$, namely
\begin{equation}\label{}
  \mathcal{F}_n = \sigma\{Z_k \mid 0\leq k \leq n\},\quad n\geq 0.
\end{equation}
The following theorem justifies Definition~\ref{def-4-1}.

\begin{theorem}\label{thr-4-1}
Suppose $(\xi_n)_{n\geq 1}\subset \mathcal{L}^2(M)$ is a martingale with respect to filtration $\mathfrak{F}=(\mathcal{F}_n)_{n\geq 0}$.
Then $(\xi_n)_{n\geq 1}$ is an $M$-generalized martingale.
\end{theorem}

\begin{proof}
By the assumptions, $(\xi_n)_{n\geq 1}$ satisfies that the following conditions
\begin{equation}\label{eq-4-5}
  \xi_n = \mathbb{E}[\xi_{n+1}\,|\, \mathcal{F}_n],\quad n\geq 0,
\end{equation}
where $\mathbb{E}[\,\cdot \mid\! \mathcal{F}_n]$ means the conditional expectation given $\sigma$-algebra $\mathcal{F}_n$.
Note that
\begin{equation*}
\mathbb{E}[Z_{\tau}\,|\, \mathcal{F}_n] = \mathbf{I}_{n]}(\tau)Z_{\tau},\quad  \tau\in \Gamma,
\end{equation*}
which, together with (\ref{eq-4-5}) and the expansion
$\xi_{n+1} = \sum_{\tau \in \Gamma}\langle Z_{\tau}, \xi_{n+1}\rangle Z_{\tau}$,
gives
\begin{equation*}
  \xi_n = \mathbb{E}[\xi_{n+1}\,|\, \mathcal{F}_n]
        = \sum_{\tau \in \Gamma}\langle Z_{\tau}, \xi_{n+1}\rangle \mathbb{E}[Z_{\tau}\,|\, \mathcal{F}_n]
        = \sum_{\tau \in \Gamma}\langle Z_{\tau}, \xi_{n+1}\rangle \mathbf{I}_{n]}(\tau)Z_{\tau}.
\end{equation*}
Taking Fock transforms yields
\begin{equation*}
  \widehat{\xi_n}(\sigma)
        = \sum_{\tau \in \Gamma}\langle \xi_{n+1}, Z_{\tau}\rangle \mathbf{I}_{n]}(\tau)\widehat{Z_{\tau}}(\sigma)
        = \langle  \xi_{n+1}, Z_{\sigma}\rangle \mathbf{I}_{n]}(\sigma)
        = \mathbf{I}_{n]}(\sigma)\widehat{\xi_{n+1}}(\sigma),
\end{equation*}
where $\sigma \in \Gamma$. Thus $(\xi_n)_{n\geq 1}$ is an $M$-generalized martingale.
\end{proof}

The next theorem gives a necessary and sufficient condition in terms of the Fock transform for an $M$-generalized martingale to be strongly convergent.

\begin{theorem}\label{thr-4-2}
Let $(\Phi_n)_{n\geq 1} \subset \mathcal{S}^*(M)$ be an $M$-generalized martingale. Then the following two conditions are equivalent:
\begin{enumerate}
  \item[(1)] $(\Phi_n)_{n\geq 1}$ is strongly convergent in $\mathcal{S}^*(M)$;
  \item[(2)] There are constants $C\geq 0$ and $p\geq 0$ such that
\begin{equation}\label{}
   \sup_{n\geq 1}|\widehat{\Phi_n}(\sigma)| \leq C\lambda_{\sigma}^p,\quad \sigma \in \Gamma.
\end{equation}
\end{enumerate}
\end{theorem}

\begin{proof}
By Theorem~\ref{thr-3-1}, we need only to prove ``$(2)\Rightarrow (1)$''.
%that $\widehat{\Phi_n}(\sigma)\rightarrow \widehat{\Phi}(\sigma)$ for all $\sigma\in \Gamma$.
Let $\sigma \in \Gamma$ be taken. Then by the definition of $M$-generalized martingales (see Definition~\ref{def-4-1}) we have
\begin{equation*}
  \widehat{\Phi_m}(\sigma) = \mathbf{I}_{m]}(\sigma)\widehat{\Phi_{m+k}}(\sigma),\quad m,\, k\geq 0.
\end{equation*}
Now take $n_0\geq 0$ such that $\sigma\in \Gamma_{n_0]}$. Then $\mathbf{I}_{{n_0}]}(\sigma)=1$ and moreover
\begin{equation*}
  \widehat{\Phi_{n_0}}(\sigma)
   = \mathbf{I}_{{n_0}]}(\sigma)\widehat{\Phi_{n}}(\sigma)
   = \widehat{\Phi_{n}}(\sigma),\quad n>n_0,
\end{equation*}
which implies $\big(\widehat{\Phi_n}(\sigma)\big)$ converges. Thus, by Theorem~\ref{thr-3-2}, $(\Phi_n)_{n\geq 1}$ is strongly convergent in $\mathcal{S}^*(M)$.
\end{proof}

\begin{theorem}\label{thr-4-3}
Let $D$  be a subset of $\mathcal{S}^*(M)$. Then the following two conditions are equivalent:
\begin{enumerate}
  \item[(1)] There is a constant $p\geq 0$ such that $D$ is contained and bounded in $\mathcal{S}_p^*(M)$;
  \item[(2)] There are constants $C\geq 0$ and $p\geq 0$ such that
\begin{equation}\label{}
   \sup_{\Phi \in D}|\widehat{\Phi}(\sigma)| \leq C\lambda_{\sigma}^p,\quad \sigma \in \Gamma.
\end{equation}
\end{enumerate}
\end{theorem}

\begin{proof}
Obviously, condition (1) implies condition (2). We now verify the inverse implication relation. In fact, under condition (2), by using Lemma~\ref{lem-2-6} we have
\begin{equation*}
 \sup_{\Phi \in D} \|\Phi\|_{-q} \leq C\bigg[\sum_{\sigma \in \Gamma}\lambda_{\sigma}^{-2(q-p)}\bigg]^{\frac{1}{2}},
\end{equation*}
where $q>p+\frac{1}{2}$, which clearly implies condition (1).
\end{proof}

The next theorem shows that for an $M$-generalized martingale, its strong (weak) convergence is just equivalent to its strong (weak) boundedness.

\begin{theorem}\label{thr-4-4}
Let $(\Phi_n)_{n\geq 1} \subset \mathcal{S}^*(M)$ be an $M$-generalized martingale. Then the following conditions are equivalent:
\begin{enumerate}
  \item[(1)] $(\Phi_n)_{n\geq 1}$ is strongly convergent in $\mathcal{S}^*(M)$;
  \item[(2)] $(\Phi_n)_{n\geq 1}$ is weakly bounded in $\mathcal{S}^*(M)$;
  \item[(3)] $(\Phi_n)_{n\geq 1}$ is strongly bounded in $\mathcal{S}^*(M)$;
  \item[(4)] $(\Phi_n)_{n\geq 1}$ is bounded in $\mathcal{S}_p^*(M)$ for some $p\geq 0$.
\end{enumerate}
\end{theorem}

\begin{proof}
Clearly, conditions (2), (3) and (4) are equivalent each other because $\mathcal{S}(M)$ is a nuclear space (see Lemma~\ref{lem-2-3}).
Using Theorems~\ref{thr-4-2} and \ref{thr-4-3}, we immediately know that conditions (1) and (4) are also equivalent.
\end{proof}

\section{Applications}

In the last section we show some applications of our main results.

Recall that the system $\{Z_{\sigma}\mid \sigma\in \Gamma\}$ is an orthonormal basis of $\mathcal{L}^2(M)$. Now if we write
\begin{equation}\label{eq-5-1}
  \Psi_n^{(0)} = \sum_{\tau \in \Gamma\!_{n]}}Z_{\tau},\quad n\geq 0,
\end{equation}
then $\big(\Psi_n^{(0)}\big)_{n\geq 0} \subset \mathcal{L}^2(M)$, and moreover $\big(\Psi_n^{(0)}\big)_{n\geq 0}$ is a martingale with respect to filtration
$\mathfrak{F}=(\mathcal{F}_n)_{n\geq 0}$. However,  $\big(\Psi_n^{(0)}\big)_{n\geq 0}$ is not convergent in $\mathcal{L}^2(M)$ since
\begin{equation}\label{eq-5-2}
  \|\Psi_n^{(0)}\|= \sqrt{\#(\Gamma\!_{n]})} = 2^{\frac{n+1}{2}} \rightarrow \infty\quad (\text{as $n\rightarrow \infty$}),
\end{equation}
where $\#(\Gamma\!_{n]})$ means the cardinality of $\Gamma\!_{n]}$ as a set and $\|\cdot\|$ the norm in $\mathcal{L}^2(M)$.

\begin{proposition}\label{prop-5-1}
The sequence $\big(\Psi_n^{(0)}\big)_{n\geq 0}$ defined above is an $M$-generalized martingale, and moreover it is strongly convergent in $\mathcal{S}^*(M)$.
\end{proposition}

\begin{proof}
According to Theorem~\ref{thr-4-1}, $\big(\Psi_n^{(0)}\big)_{n\geq 0}$ is certainly an $M$-generalized martingale.
On the other hand, in viewing the relation between the canonical bilinear form on $\mathcal{S}^*(M)\times \mathcal{S}(M)$ and the inner product in $\mathcal{L}^2(M)$, we have
\begin{equation}\label{eq-5-3}
  \widehat{\Psi_n^{(0)}}(\sigma)
   = \langle\!\langle \Psi_n^{(0)}, Z_{\sigma} \rangle\!\rangle
   = \langle \Psi_n^{(0)}, Z_{\sigma} \rangle
   = \mathbf{I}_{n]}(\sigma),\quad \sigma \in \Gamma,\, n\geq 0,
\end{equation}
which implies that
\begin{equation*}
  \sup_{n\geq 0}\big|\widehat{\Psi_n^{(0)}}(\sigma)\big| \leq C \lambda_{\sigma}^p,\quad \sigma \in \Gamma
\end{equation*}
with $C=1$ and $p=0$. It then follows from Theorem~\ref{thr-4-2} that $\big(\Psi_n^{(0)}\big)_{n\geq 0}$ is strongly convergent in $\mathcal{S}^*(M)$.
\end{proof}

Recall that \cite{wang-chen}, for two generalized functionals $\Phi_1$, $\Phi_2 \in \mathcal{S}^*(M)$, their convolution $\Phi_1\ast\Phi_2$ is defined by
\begin{equation}\label{}
 \widehat{ \Phi_1\ast\Phi_2}(\sigma) = \widehat{ \Phi_1}(\sigma)\widehat{\Phi_2}(\sigma),\quad \sigma \in \Gamma.
\end{equation}

The next theorem provides a method to construct an $M$-generalized martingale through the $M$-generalized martingale$\big(\Psi_n^{(0)}\big)_{n\geq 0}$ defined in (\ref{eq-5-1}).

\begin{theorem}\label{thr-5-1}
Let $\Phi \in \mathcal{S}^*(M)$ be a generalized functional and define
\begin{equation}\label{}
  \Phi_n =  \Psi_n^{(0)} \ast\Phi,\quad n\geq 0.
\end{equation}
Then $(\Phi_n)_{n\geq 0}$ is an $M$-generalized martingale, and moreover it converges strongly to $\Phi$ in $\mathcal{S}^*(M)$.
\end{theorem}

\begin{proof}
By Lemma~\ref{lem-2-6}, there exist some constants $C\geq 0$ and $p\geq 0$ such that
\begin{equation}\label{eq-5-6}
  |\widehat{\Phi}(\sigma)| \leq C\lambda_{\sigma}^p,\quad \sigma \in \Gamma.
\end{equation}
On the other hand, by using (\ref{eq-5-3}), we get
\begin{equation}\label{eq-5-7}
  \widehat{\Phi_n}(\sigma)
  =\widehat{\Psi_n^{(0)}}(\sigma)\widehat{\Phi}(\sigma)
  =\mathbf{I}_{n]}(\sigma)\widehat{\Phi}(\sigma),\quad \sigma \in \Gamma,\, n\geq 0,
\end{equation}
which, together with the fact $\mathbf{I}_{n]}(\sigma)\mathbf{I}_{n+1]}(\sigma)=\mathbf{I}_{n]}(\sigma)$, gives
\begin{equation*}
  \widehat{\Phi_n}(\sigma)= \mathbf{I}_{n]}(\sigma)\widehat{\Phi_{n+1}}(\sigma),\quad \sigma \in \Gamma,\, n\geq 0.
\end{equation*}
Thus $(\Phi_n)_{n\geq 0}$ is an $M$-generalized martingale. Additionally, it easily follows from (\ref{eq-5-6}) and (\ref{eq-5-7}) that
$\widehat{\Phi_n}(\sigma)\rightarrow \widehat{\Phi}(\sigma)$ for each $\sigma \in \Gamma$ and
\begin{equation*}
  \sup_{n\geq 0}|\widehat{\Phi_n}(\sigma)|
    =\sup_{n\geq 0}\big[\mathbf{I}_{n]}(\sigma)\big]|\widehat{\Phi}(\sigma)|
  \leq C\lambda_{\sigma}^p,\quad \sigma \in \Gamma.
\end{equation*}
Therefore, by Theorem~\ref{thr-3-2}, we finally find $(\Phi_n)_{n\geq 0}$ converges strongly to $\Phi$.
\end{proof}

\section*{Appendix}

In this appendix, we provide some basic notions and facts about discrete-time normal martingales. For details we refer to \cite{privault,wang-lc}.

Let $(\Omega, \mathcal{F}, P)$ be a given probability space with $\mathbb{E}$ denoting the expectation with respect to $P$.
We denote by $\mathcal{L}^{2}(\Omega, \mathcal{F}, P)$ the usual Hilbert space of
square integrable complex-valued functions on $(\Omega, \mathcal{F}, P)$
and use $\langle\cdot,\cdot\rangle$ and $\|\cdot\|$ to mean its inner product and norm, respectively.
\vskip2mm

\noindent {\bf Definition A.1.}\label{def-A-1}
A stochastic process $M=(M_n)_{n\in \mathbb{N}}$ on
$(\Omega, \mathcal{F}, P)$ is called a discrete-time normal martingale if it is square integrable and satisfies:
\begin{enumerate}[(i)]
  \item $\mathbb{E}[M_0 | \mathcal{F}_{-1}] = 0$ and $\mathbb{E}[M_n | \mathcal{F}_{n-1}] = M_{n-1}$
 for $n\geq 1$;
  \item $\mathbb{E}[M_0^2 | \mathcal{F}_{-1}] = 1$ and $\mathbb{E}[M_n^2 | \mathcal{F}_{n-1}] = M_{n-1}^2 +1$
for $n\geq 1$,
\end{enumerate}
where $\mathcal{F}_{-1}=\{\emptyset, \Omega\}$,
$\mathcal{F}_n = \sigma(M_k; 0\leq k \leq n)$ for $n \in \mathbb{N}$ and $\mathbb{E}[\cdot | \mathcal{F}_{k}]$ means the conditional expectation.
\vskip2mm

Let $M=(M_n)_{n\in \mathbb{N}}$ be a discrete-time normal martingale on $(\Omega, \mathcal{F}, P)$.
Then one can construct from $M$ a process $Z=(Z_n)_{n\in \mathbb{N}}$ as

\begin{equation*}
Z_0=M_0,\quad Z_n = M_n-M_{n-1},\quad  n\geq 1. \eqno(\mathrm{A}.1)
\end{equation*}
It can be verified that $Z$ admits the following properties:
\begin{equation*}
    \mathbb{E}[Z_n | \mathcal{F}_{n-1}] =0\quad \text{and}\quad  \mathbb{E}[Z_n^2 | \mathcal{F}_{n-1}] =1,\quad  n\in \mathbb{N}. \eqno(\mathrm{A}.2)
\end{equation*}
Thus, it can be viewed as a discrete-time noise.
\vskip2mm

\noindent {\bf Definition A.2.} \label{def-A-2}
Let $M=(M_n)_{n\in \mathbb{N}}$ be a discrete-time normal martingale. Then the process $Z$ defined by (A.2) is called the discrete-time normal noise associated
with $M$.
\vskip2mm

The next lemma shows
that, from the discrete-time normal noise $Z$, one can get an orthonormal system in $\mathcal{L}^2(\Omega, \mathcal{F}, P)$,
which is indexed by $\sigma \in \Gamma$.
\vskip2mm

\noindent {\bf Lemma A.1.}
Let $M=(M_n)_{n\in \mathbb{N}}$ be a discrete-time normal martingale and $Z=(Z_n)_{n\in \mathbb{N}}$ the discrete-time normal noise associated with $M$.
Define $Z_{\emptyset}=1$, where $\emptyset$ denotes the empty set, and
\begin{equation*}
    Z_{\sigma} = \prod_{i\in \sigma}Z_i,\quad \text{$\sigma \in \Gamma$, $\sigma \neq \emptyset$}. \eqno(\mathrm{A}.3)
\end{equation*}
Then $\{Z_{\sigma}\mid \sigma \in \Gamma\}$ forms a countable orthonormal system in
$\mathcal{L}^2(\Omega, \mathcal{F}, P)$.
\vskip2mm

Let $\mathcal{F}_{\infty}=\sigma(M_n; n\in \mathbb{N})$, the $\sigma$-field over $\Omega$ generated by $M$.
In the literature, $\mathcal{F}_{\infty}$-measurable functions on $\Omega$ are also known as functionals of $M$.
Thus elements of $\mathcal{L}^2(\Omega, \mathcal{F}_{\infty}, P)$ can be called square integrable functionals of $M$.
For brevity, we usually denote by $\mathcal{L}^2(M)$ the space of square integrable functionals of $M$, namely
\begin{equation*}
\mathcal{L}^2(M)=\mathcal{L}^2(\Omega, \mathcal{F}_{\infty}, P). \eqno(\mathrm{A}.4)
\end{equation*}
\vskip2mm

\noindent {\bf Definition A.3.} \label{def-A-3}
The discrete-time normal martingale $M$ is said to have the chaotic representation property if
the system $\{Z_{\sigma}\mid \sigma \in \Gamma\}$ defined by (A.3) is total in $\mathcal{L}^2(M)$.
\vskip2mm

So, if the discrete-time normal martingale $M$ has the chaotic representation property, then the system $\{Z_{\sigma}\mid \sigma \in \Gamma\}$
is actually an orthonormal basis for $\mathcal{L}^2(M)$,
which is a closed subspace of $\mathcal{L}^2(\Omega, \mathcal{F}, P)$ as is known.
\vskip2mm

\noindent{\bf Remark A.1.}
\'{E}mery \cite{emery} called a $\mathbb{Z}$-indexed process $X=(X_n)_{n \in \mathbb{Z}}$ satisfying (A.2)
a novation and introduced the notion of the chaotic representation property for such a process.

\section*{Acknowledgement}

This work is supported by National Natural Science Foundation of China (Grant No. 11461061).

\end{document}